\documentclass[11pt,a4paper]{amsart}

\usepackage{color}
\usepackage{xcolor}

\usepackage[utf8]{inputenc}
\usepackage[T1]{fontenc} 
\usepackage{geometry}        
\usepackage[french, english]{babel}
\usepackage{amsmath}
\usepackage{amsfonts}
\usepackage{amsthm}

\usepackage{mathrsfs}

\usepackage{amscd}
\usepackage{hyperref}
\usepackage{xcolor}
\usepackage{enumitem}

\usepackage{graphicx}

\usepackage{pstricks,pst-plot}

   \usepackage{amssymb}
   \usepackage{verbatim}
   \newif\ifpdf
   \ifpdf
     \usepackage[pdftex]{graphicx}
     \usepackage[pdftex]{hyperref}
   \else
     \usepackage{graphicx}
   \fi

\usepackage[all]{xy}

\newcommand{\R}{\mathbb{R}}

\newcommand{\N}{\mathbb{N}}
\newcommand{\Z}{\mathbb{Z}}
\newcommand{\C}{\mathbb{C}}

\newcommand{\D}{\mathbb{D}}

\newtheorem{thm}{Theorem}[section]
\newtheorem{prop}[thm]{Proposition}
\newtheorem{coro}[thm]{Corollary}
\newtheorem{defi}[thm]{Definition}
\newtheorem{lem}[thm]{Lemma}

\newtheorem{rem}[thm]{Remark}

\newcommand{\rs}{\mathbb{P}^1}

\newcommand{\limn}{\lim_{n \rightarrow \infty}}
\newcommand{\id}{\mathrm{Id}}

\newcommand{\bcal}{\mathcal{B}}

\newcommand{\lcal}{\mathcal{L}}

\newcommand{\eps}{\epsilon}

\newcommand{\la}{\lambda}
\newcommand{\lam}{\lambda}

\numberwithin{equation}{section}

\title{Horn maps of semi-parabolic
Hénon maps}
\begin{author}[M.~Astorg]{Matthieu Astorg}
\address{Université d’Orléans, Institut Denis
Poisson, UMR CNRS 7013, 45067 Orléans Cedex 2, France}
\email{matthieu.astorg@univ-orleans.fr}
\end{author}

\begin{author}[F.~Bianchi]{Fabrizio Bianchi}
\address{Dipartimento di Matematica, Università di Pisa, Largo Bruno Pontecorvo 5, 56127 Pisa, Italy}
  \email{fabrizio.bianchi$@$unipi.it}
\end{author}

\begin{document}

\maketitle
\selectlanguage{english}

\begin{abstract}
	We prove that horn maps associated to quadratic semi-parabolic fixed points of Hénon maps, first introduced by Bedford, Smillie, and Ueda,
 satisfy a weak form of the Ahlfors island property. As a consequence,
 two natural
 definitions of
 their Julia set (the non-normality locus of the family of iterates and
 the closure of the set of the repelling periodic points) coincide.
    As another consequence, we also prove that there exist small perturbations of semi-parabolic Hénon maps for which
    the Hausdorff dimension of the
    forward Julia set $J^+$ is
     arbitrarily close to $4$.
\end{abstract}

\section{Introduction}

Following the seminal work of
Douady-Hubbard \cite{DH84-p1,DH85-p2}
 and
 Lavaurs \cite{lavaurs1989systemes},  
the study of perturbations of maps
with a parabolic fixed point,
often referred to as \emph{parabolic implosion}, has
been a major theme in one-variable complex dynamics. A first consequence of this theory
is the discontinuity of the Julia sets
at parameters where a parabolic bifurcation occurs \cite{Douady94Julia}.
Further
notable applications include  the celebrated result by Shishikura \cite{Shishikura98hausdorff} that
parabolic maps can be approximated by hyperbolic maps
with arbitrarily large 
hyperbolic dimension
and, as a consequence, the proof that the boundary of the Mandelbrot set has maximal Hausdorff dimension, see also
\cite{DNS97, mcmullen, Tan98hausdorff,Zinsmeister98}
for refinements of this result and further consequences
 and
\cite{Shishikura00bifurcation,PV20}
 for an overview of the theory. 

Shishikura's proof involves 
the so-called \emph{horn maps}. These maps describe the limit behaviour
of the return maps of large iterates of the perturbed maps near the parabolic point.
Another important aspect of horn maps
 (also called \emph{\'Ecalle-Voronin invariants})
 is that they classify parabolic germs up to analytic conjugacy, i.e., they form a \emph{complete invariant} for this notion of equivalence \cite{Ecalle85-tome3,Voronin81analytic}.
Techniques from parabolic implosion
have been extended by Inou and Shishikura \cite{inou2006renormalization}, who introduced the \emph{near-parabolic renormalization}, a powerful tool which was used,
for instance,
to construct quadratic Julia sets with positive area \cite{BC12area},
and to make
significant 
steps towards the settling of the Fatou hyperbolicity conjecture
\cite{cheraghi2015satellite}.

In recent years, techniques of parabolic implosion  have started to be developed and  successfully
applied also in higher dimensions. 
In \cite{BSU17semi}, Bedford, Smillie, and Ueda extended Lavaurs' results to diffeomorphisms of $(\C^2,0)$ 
with a semi-parabolic fixed point, that is, with one multiplier equal to $1$ and one in the unit disk. In the important particular case of dissipative Hénon map, the authors introduced 
a horn map analogous to the one-dimensional case, which they used to show the discontinuity of various dynamically meaningful sets
at parameters with a semi-parabolic fixed point.
In \cite{DL15stability}, 
Dujardin and Lyubich 
adapted and improved
the results from \cite{BSU17semi}
to construct homoclinic tangencies in some regions of the parameter space of complex Hénon maps, as part of their
characterization of stability and bifurcation
in families
of such maps.
Parabolic implosion
techniques in two complex 
variables were also used
by Buff, Dujardin, Peter, Raissy, and the first author 
to give the first example of 
an endomorphism of $\mathbb P^2(\mathbb C)$
with a wandering domain \cite{ABDPR16}, which solved a long-standing open question in the domain, see also
\cite{ABP23}.
Adaptations of these techniques from Boc Thaler and  the first author have lead to a precise description of the local dynamics near a parabolic point of a significant class of maps tangent to the identity in $\mathbb C^2$
\cite{abate2001residual},
in particular solving a long-standing open question by Abate \cite{abate2005classification}
on the topological classification of such maps. 
The first result on the parabolic implosion of a two-dimensional map tangent to the identity was also established by the second author in \cite{bianchi19parabolic}.

\medskip

Coming back to the original work 
by Bedford, Smillie, and Ueda, 
very little is known about the dynamics of the horn maps of semi-parabolic Hénon maps; 
for instance, 
it was not even known until now whether they always had periodic 
cycles 
(besides the two "trivial" fixed points $0$ and $\infty$).
In this paper, 
we prove that they satisfy a weak version of the so-called Ahlfors island property. As a consequence, we can show the
density of the repelling periodic points in  their Julia sets, and an analogous
of Shishikura's result
for the forward Julia set $J^+$ of dissipative Hénon maps.

\medskip

The class of \emph{Ahlfors island maps}
was implicitly present in the work of Epstein \cite{epstein1993towers}, see also \cite{lasserippon,lassevolker}. Roughly speaking,
given an open set $W\subset \rs$, a holomorphic map $f:W \to \rs$  has the { $N$ islands} property if,
given any $N$ Jordan domains with pairwise disjoint closures, one can find univalent inverse branches of $f$ on at least one of these Jordan domains, whose image is close to any given point in the boundary of $W$ (see Definition \ref{d:island} for a precise formulation).
A celebrated theorem 
by Ahlfors \cite{Ahlfors1935theorie}
states that every entire or meromorphic map has the 
{ $5$ islands} property (in this case, we have $W=\C$ and the only boundary point is $\infty$), see also
\cite{Bergweiler1998ahlfors}. This remarkable result can be used to give 
a simple proof 
of the density of repelling cycles in the Julia set of any transcendental entire or meromorphic map,
see for instance
\cite{Bergweiler93iteration}.

Another class of maps extensively studied by Epstein \cite{epstein1993towers} 
is the class of 
\emph{finite type maps}, that is, holomorphic maps with finitely many singular values (see Definition \ref{d:finite-type}). 
Finite type maps with $N$ singular values have the 
 $(N+1)$ islands
property; however, finite type maps form a much smaller  class than Ahlfors island maps. For instance, Epstein proved that finite type maps have no wandering or Baker domains 
and admit only finitely many non-repelling cycles; none of these statements is true in general for Ahlfors island maps. While horn maps of one-variable rational maps are always finite type maps
  \cite{epstein1993towers}, 
it seems unlikely that this holds true for horn maps of 
dissipative semi-parabolic Hénon maps in general
(indeed, this would be equivalent to proving that only finitely many stable manifolds have tangencies with a certain entire curve $\Sigma$, defined below).

\medskip

In this paper, 
we introduce the following slightly weaker version of the island property.
Observe that 
in the
usual versions,
the property  below is true for every
positive real number 
instead of just those smaller than
$r(z_0)$, see Definition \ref{d:island}.
On the other hand, here
we can specify which points $z_0$ should be excluded,
see Remark \ref{r:island-small-island}.

\begin{defi}\label{d:small-island}
	Let  $ W \subset \rs$ be an open set
and $h: W \to \rs$ 
a holomorphic map. We say that $h$ has the \emph{small island property}
if,
for every $z_0 \in \C^*$, 
	there exists $r(z_0)>0$ such that, for every domain $U$ intersecting $\partial W$, there exists $\Omega \Subset U \cap W$ such that 
	$h: \Omega \to \D(z_0,r(z_0))$ is a conformal isomorphism. 
\end{defi}

Let us emphasize that $r(z_0)$ does not depend on the choice of $U$, but only on $f$ and $z_0$.
We will show in Theorem \ref{th:islandp}
that the small island property as in Definition \ref{d:small-island} is enough to prove
the density of the repelling periodic points in the Julia set.

\medskip

Let now
$f$ be a dissipative Hénon map with a semi-parabolic fixed point $p$ of order $2$, see Section \ref{s:prelim} 
for the precise definitions.
Let $\mathcal B$ denote the parabolic basin of the semi-parabolic point, which is a two-dimensional open set with 
$p$ on its boundary, and 
$\Sigma$ the parabolic curve,
which can thought of as the one-dimensional unstable curve of $p$.
By 
\cite{BSU17semi}, 
there exist
two maps 
$\phi^\iota\colon {\mathcal B}\to \C$
(usually called the 
\emph{incoming Fatou coordinate}) 
 and
$\psi^o: \C \to \Sigma$
(usually called the 
\emph{outgoing Fatou parametrization})
which semi-conjugate
the map 
$f$ 
on $\mathcal B$ and on $\Sigma$
to the translation by $1$
on $\C$, respectively.
The
\emph{H\'enon-Lavaurs map}
(or \emph{transition map})
associated to $f$
is the composition
$\mathcal L_0:=\psi^o \circ \phi^\iota$. 
Observe that $\mathcal L_0$
commutes with $f$.
By \cite{BSU17semi}, $\mathcal L_0$
can be seen as a limit of large iterates of suitable perturbations of $f$ near the semi-parabolic point.
We refer to Section \ref{ss:prelim-parabolic}
for more details on this.

By the definition of $\mathcal L_0$, 
we immediately see that this map is semi-conjugated to the map
$H_f := \phi^\iota \circ \psi^o\colon (\psi^o)^{-1} (\mathcal B)\to \C$.
This 
is  the \emph{lifted horn map}
associated to $f$.
As 
$\mathcal L_0$ commutes with $f$,
 we see that $H_f$ commutes with the translation by $1$ on $\mathbb C$. We can then quotient its action by this translation, and obtain a map from 
 a subset of the cylinder  $\mathbb C^*$,
  containing pointed neighbourhoods of $0$ and $\infty$,
 to the cylinder itself. By \cite{BSU17semi}, $h_f$
  extends to $0$ and
 $\infty$. The following is our first main result.

\begin{thm}\label{th:islandp}
Let $f$ be a dissipative semi-parabolic Hénon map as above.
	The horn map $h_f$ has the small island property as in Definition \ref{d:small-island}.
\end{thm}

The following is then a
consequence of Theorem
\ref{th:islandp}.

\begin{coro}
Let $f$ be a dissipative semi-parabolic Hénon map as above.
	The repelling periodic points of the horn map $h_f$ are dense in its Julia set.
\end{coro}

The  proof of Theorem \ref{th:islandp} is 
based on techniques from Pesin theory, which were first adapted to this context in \cite{BLS93measure}, and on local equidistribution results towards the Green currents $T^+$ and $T^-$
of the Hénon map $f$, which follow
from the local approch to these problems
developed in \cite{HOV95,Duj04henon,DS06geometry,DNS08dynamics}.

\medskip

As mentioned above,
as 
an application of  Theorem \ref{th:islandp},
we will also deduce an analogous of Shishikura's result for Hénon maps.

\begin{thm}\label{th:maxdim}
    Let $f$ be a dissipative semi-parabolic Hénon map as above. 
Then, 
    there exists a sequence $f_n \to f$ 
    of dissipative Hénon maps of the same algebraic degree 
    such that 
    $$\dim_H J^+(f_n) \to 4 \quad \mbox{ as } n\to \infty.$$
\end{thm}

In order to prove
Theorem \ref{th:maxdim}, 
we show that every holomorphic map with the small island property can be suitably modified (i.e., can be multiplied by a suitable constant)
so that
it has
arbitrarily
large
(i.e., close to 2)
 hyperbolic dimension. 
Assume for the 
sake of simplicity that the multiplication is not necessary
(this is just a minor technical point).
We can
apply this to the horn map $h_f$, and recall
that the H\'enon-Lavaurs map $\mathcal L_0$
is the limit of suitable large iterates $f^n_{\epsilon_n}$
of perturbations $f_{\epsilon_n}$
of $f$.
By the conjugacy, large hyperbolic sets for $h_f$ give
a large limit set for
suitable iterates of $\mathcal L_0$.
As these large hyperbolic sets
persist under the perturbation $f\mapsto f_{\epsilon_n}$
for $n$ sufficiently large, this leads to 
hyperbolic limit
sets $\mathcal H_n$
for suitable iterates of $f_{\epsilon_n}^n$
with large (i.e., close to 2)
unstable dimension.
As the forward Julia set $J^+ (f_{\epsilon})$
contains the union of the stables manifolds of the points of $\mathcal H_n$, this leads to the desired result.

\subsection*{Acknowledgments}

The authors would like to express their gratitude to Eric Bedford for bringing this problem to their attention.
The first author would also like to thank the University of Pisa for the warm welcome and the excellent work conditions.

This project has received funding from
 the French government through the Programme
 Investissement d'Avenir
(ANR QuaSiDy /ANR-21-CE40-0016,
ANR PADAWAN /ANR-21-CE40-0012-01)
and from the MIUR Excellence Department Project awarded to the Department of Mathematics, University of Pisa, CUP I57G22000700001.
Both authors are part of the PHC Galileo project G24-123.

\section{Island properties and consequences}

\subsection{Ahlfors island maps and finite type maps}

\begin{defi}\label{def:sing}
	Let $X$  and $W$ be a Riemann surface, 
 with $X$ connected, 
 and 
 $f: W \to X$
 a holomorphic map. The \emph{singular value set} $S(f)$  of $f$
 is the smallest subset of $X$ such that $f: W_0 \setminus f^{-1}(S(f)) \to X \setminus S(f)$ is a covering map for every 
 connected component $W_0$ of $W$.
\end{defi}

As covering maps
are surjective, 
it follows from the above definition that
we have $X \setminus f(W) \subset S(f)$.

\begin{defi}\label{d:finite-type}\cite{epstein1993towers}
	Let $W \subset \rs$ be a non-empty open set and $f: W \to \rs$ 
 a holomorphic map. We say that $f$ is a 
 \emph{finite type map
 on $\rs$} if 
	\begin{enumerate}
		\item $f$ is non-constant on every connected component of $W$;
		\item $f$ has no removable singularities;
		\item $S(f)$ is finite.
	\end{enumerate}
\end{defi}

\begin{defi}\label{d:island}\cite{lasserippon, lassevolker}
    Let $W \subset \rs$ be a non-empty open set and  $f: W \to \rs$
    a holomorphic map.
    We say that $f$ has the
    \emph{$N$ islands property} 
    if,
    given any $N$ Jordan domains $D_1, \ldots, D_N \subset \rs$ with pairwise disjoint
    closures 
    and any open set $U$ intersecting $\partial W$, there exists $1 \leq i_0 \leq N$ and an open set $\Omega \Subset U \cap W$ such that $f: \Omega \to D_{i_0}$ is a conformal isomorphism.
 If there exists $N \geq 1$ such that $f$ has the $N$ islands property, we say that $f$ is an \emph{Ahlfors island map}.
\end{defi}

\begin{rem}\label{r:island-small-island}
In 
Definition \ref{d:small-island}, 
$z_0$ must be chosen different from $0$ and $\infty$. 
Hence, the small island property
as in that definition can be seen as a weaker version of the 
{ 3 islands} property above. 
We could give a more general definition of the small { $N$ islands} property, admitting $N-1$ exceptions as in the Definition \ref{d:island}. The proofs in this section
would be the same. We just define the precise version
of the property that we will prove for 
the horn maps 
of semi-parabolic Hénon maps
for simplicity.
\end{rem}

Also note that the $1$ island property is vacuously satisfied if $\partial W=\emptyset$, that is, if $f: \rs \to \rs$ is a rational map. This case is however very special, and our arguments will often use the fact that $\partial W \neq \emptyset$.

\subsection{Julia sets}

We fix in this section 
an open set 
$W \subset \rs$
and a holomorphic map
 $f: W \to \rs$. We give here two natural definitions of the Julia set of $f$.
The first 
is related to the notion of non-normality.
Since $W \neq \rs$, 
the definition needs to take into account
orbits leaving
the domain $W$.

\begin{defi}[Definition of the Julia set as non-normality locus]\label{d:julia-normal}
   The Fatou set $F(f)$
   of $f$
   is the union of all open sets $U \subset  \rs$ such that either
    \begin{enumerate}
    \item $f^n(U) \subset W$ for all $n \in \N$, and $\{f^n: U \to W\}$
    is normal; or 
    \item there exists $n \in \N$ such that $f^n(U) \subset \rs \setminus \overline{W}$, where $\overline{W}$ denotes the closure of $W$ in $\rs$.
    \end{enumerate}
    The set $J_F(f)$ is the complement of  $F(f)$
    in $\rs$.
\end{defi}

Observe that according to this definition, we always have $\partial W \subset J_F(f)$
and  $\rs \setminus \overline{W} \subset F(f)$.
Given a point in 
the Fatou set, all its orbit is in the Fatou set; conversely, given a point in
$J_F (f)$, all preimages are in $J_F(f)$.
Moreover, it is clear that repelling periodic points are always in $J_F(f)$. This leads to the second definition, and to the inclusion $J_R(f) \subset J_F(f)$.

\begin{defi}[Definition of the Julia set by means of repelling periodic points]\label{d:julia-rep}
 The set  $J_R(f)$ is the closure of the set of all repelling periodic points of $f$.
\end{defi}

\subsection{A consequence 
of the small island property}
We fix again  
an open set 
$W \subset \rs$
and a holomorphic map
 $f: W \to \rs$. 
By \cite{epstein1993towers}, we have $J_F(f)=J_R(f)$
if $f$ 
is an Ahlfors island map.
In this section 
we show that the small island property as in Definition \ref{d:small-island} 
is still enough to ensure this property.
For simplicity, we will denote by
$W_\infty$ the interior of 
$\bigcap_{n \geq 0} f^{-n}(W)$
in the rest of this section.

\begin{lem}\label{lem:boundary}
Assume that 
either $W_\infty=\emptyset$ or all connected components of $W_\infty$ are
hyperbolic.
Then we have
	$J_F(f)=\overline{\bigcup_{n \geq 0} f^{-n}(\partial W)}$.
\end{lem}

\begin{proof}
	The inclusion $\overline{\bigcup_{n \geq 0} f^{-n}(\partial W)} \subset J_F(f)$ is always true by the Definition \ref{d:julia-normal} of $J_F(f)$.

	Conversely,  if $W_\infty\neq \emptyset$, $W_\infty$ is completely invariant 
	and $f: W_\infty \to W_\infty$ is non-increasing for the hyperbolic metric.
 Therefore, we have $W_\infty \subset F(f)$. Clearly, this inclusion still holds if $W_\infty=\emptyset$.

Let $U$ be a connected open set intersecting $J_F(f)$. By the inclusion proved above, $U$ cannot be contained in $W_\infty$.
Therefore, there exists $n \in \N$ such that $f^n(U) \cap (\rs \setminus W) \neq \emptyset$.
Moreover, we must have $f^n(U) \cap \partial W \neq \emptyset$, for otherwise we would have $f^n(U) \subset \rs \setminus W$. By definition, this would give
$U \subset F(f)$, contradicting 
the assumption 
$U\cap J_F (f)\neq \emptyset$.

The above proves that, for any open domain $U$ intersecting $J_F(f)$, there exists $n \in \N$
such that $f^{-n}(\partial W) \cap U \neq \emptyset$.
Again by definition, we have
$f^{-n}(\partial W)  \subset J_F(f)$.
	This shows that 
 $\bigcup_{n \geq 0} f^{-n}(\partial W)$ is indeed dense in $J_F(f)$, and concludes the proof.
\end{proof}





\begin{thm}\label{th:repdense}
	Assume that $f$ has 
 the small island
 property. 
	Then $J_F(f)= J_R(f)$.
\end{thm}

\begin{proof}
Recall that the inclusion $J_R(f) \subset J_F(f)$
always holds by definition. Hence, we only have to prove the reversed inclusion.

	If $W_\infty$ is non-empty and has a non-hyperbolic component, 
 it is isomorphic to  $\rs$, $\C^*$,  or $\C$.
 Then, $f$ is either a rational map, a transcendental self-map of $\C^*$ or a transcendental entire map.  In all these cases, the result is classical.
 Therefore, we only need to deal with the case where 
 $W_\infty$
 is either empty or
 has a hyperbolic connected component. In particular, Lemma \ref{lem:boundary} applies.
	
	Fix $z_0 \in J_F(f)$. Since $J_F(f)$ has no isolated points, 
 we may assume without loss of generality that $z_0 \in \C^*$.
 By Lemma \ref{lem:boundary}, we may also assume that $f^n(z_0) \in \partial W$ for some $n \in \N$.
	Let $D$ be a 
 small disk centered at $z_0$
 small enough so that $f^n: D \to f^n(D)$ is a branched cover, ramified only possibly at $f^n(z_0)$.
 We can also assume that the radius of $D$ is smaller than the quantity $r(z_0)$ as in the Definition \ref{d:small-island}
of the small island property. It is enough to show that there exists a repelling point in $D$.

Let $\Omega \Subset f^n(D)$ be the simply connected domain given by the small island property applied with $U=f^n (D)$. Then, the map $f^{n+1}: f^{n} (D)\Supset \Omega \to f^{n}(D)$ is a branched cover with at most one critical value $f^n(z_0)$, which lies outside of $\overline{\Omega}$, and in particular, it is polynomial-like.
It follows from, e.g., Douady-Hubbard's Straightening theorem \cite{DH85dynamics} that 
 $f^{n+1}:  \Omega \to f^{n}(D)$ has a repelling periodic point\footnote{Actually, since this polynomial-like map has only one critical value which is escaping, it can also
 be proved by elementary means that it has a repelling fixed point, see for instance \cite{epstein1993towers}.}. 
 
 Therefore, $D$ contains a repelling periodic point for $f$.
This proves the inclusion $J_F(f) \subset J_R(f)$
and completes the proof. 
\end{proof}

\begin{rem}\label{r:natural}
A holomorphic family of maps 
$f_\la: W_\la \to \rs$ is \emph{natural} 
\cite{EL92entire,ABF} if it is 
of the form $f_\la = \phi_\la \circ f \circ \psi_\la^{-1}$, where 
$W\subset \rs$ is an open set,
$f: W \to \rs$ a holomorphic map, $W_\lam:=\psi_\la(W)$ and $\phi_\la, \psi_\la: \rs \to \rs$ 
are homeomorphisms depending holomorphically on $\la \in M$.
It is straightforward to check that if $f$ has the small island property, then 
each map $f_\la$ in 
a natural family as above
also has the small island property.
In the proof of Theorem \ref{th:maxdim}, we will be interested in the case of a family of maps of the form 
$h_\la = \la h$, where $h$ has the small island property and $\la \in \C^*$. 
It is in particular a natural family, with $\phi_\la(z):=\la z$ and $\psi_\la:=\id$.
\end{rem}

\section{Preliminaries on Hénon and horizontal-like maps}
\label{s:prelim}

We will consider in this section an automorphism $f$ of $\mathbb C^2$
of the form
\begin{equation}\label{eq:henon}
f(z,w) = (p(z)+ aw, z),
\end{equation}
where $p$ is a monic polynomial 
of degree $d\geq 2$
and $a$ is some constant in
$\mathbb C^*$. 
Any $f$ as above is usually referred to as a \emph{(generalized) Hénon map}. Observe that the Jacobian of $f$ is constant, and equal to $|a|$. We say that $f$
is \emph{dissipative} if $|a|<1$. 
\medskip

By results of Jung  \cite{Jung42birationale}
and Friedland-Milnor \cite{FM89dynamical},
every
polynomial automorphism $f$ of $\mathbb C^2$
is conjugated 
(in the group of polynomial automorphisms)
to either an \emph{elementary automorphism},
i.e., a map 
of the form
$(z,w)\mapsto (a z + p(w), bw+c)$
for some  $a,b\in \mathbb C^*$, $c\in \mathbb C$, and 
polynomial $p$ of degree $d\geq 0$,
or to a \emph{Hénon-type map}, i.e., a finite composition of generalized Hénon maps as in \eqref{eq:henon}. 
As the dynamics of elementary automorphisms is simple to describe, we will just consider in the following maps of the second type. For simplicity, we will just consider Hénon maps, but the picture is the same when considering finite compositions as above.

\subsection{Filtration property and induced horizontal-like map}\label{ss:filtration}

Let 
$f$ be as in \eqref{eq:henon}. Friedland and Milnor 
\cite{FM89dynamical}
showed that it is possible to decompose $\mathbb C^2$ in a dynamically meaningul way, as follows. Let $\mathbb D(0,R)\subset \mathbb C$ be the disc of center 0 and radius $R$. Set
\[
\begin{aligned}
D_R & := \mathbb D(0,R)^2,\\
V_R^+ & := \{(z,w)\in \mathbb C^2\colon |z|>\max (|w|,R)\}, \mbox{ and}\\
V_R^- & := \{(z,w)\in \mathbb C^2\colon |w|>\max (|z|,R)\}.
\end{aligned}
\]
We also denote
by $K^+$ (resp. $K^-$) the set of points whose  orbit under $f$ (resp. $f^{-1}$)
is bounded.

\begin{lem}
\label{lem:filtration}
The following assertions hold for every $R$ sufficiently large.
\begin{enumerate}
\item $f(V^+_R)\subset V_R^+$ and $V_R^+ \cap K^+= \emptyset$;

\item $f(D_R \cup V_R^+)\subset D_R \cup V_R^+$;
\end{enumerate}
\end{lem}
 
Similar assertions hold replacing $f$, $V_R^+$, and $K^+$ with $f^{-1}$, $V_R^-$, and $K^-$, respectively.

\medskip

Although our maps will always be globally defined,
in the following
we will sometimes need to work in a semi-local setting, that we now describe.
A \emph{vertical
subset} of $D_R$ is a subset of $D_R$ whose closure in $\mathbb C^2$
is disjoint from the 
\emph{vertical boundary} 
$\partial \mathbb D(0,R)\times \mathbb D(0,R)$ of 
$D_R$. Similarly, a
\emph{horizontal
subset} of $D_R$ is a subset of $D_R$ whose closure 
is disjoint from the \emph{horizontal
boundary} 
$\mathbb D(0,R) \times \partial \mathbb D(0,R)$.

We fix $R$ sufficiently large for Lemma \ref{lem:filtration} to hold.
The map $f$ can be seen as a map from the set 
$D_R\cap f^{-1} (D_R)$ to the set $D_R\cap f(D_R)$. In particular, since $f^{-1}(D_R)\Subset D_R\cup V^-$ and $f(D_R)\Subset D_R\cup V^+$, 
$f$ is a holomorphic map from a vertical subset of $D_R$ to a horizontal one and
it is a so-called \emph{horizontal-like map}.
We refer to \cite{HOV95,Duj04henon,DS06geometry,DNS08dynamics}
for the precise definition and their properties. 
We will use the notation $\tilde f$ 
for the horizontal-like map associated to $f$ as above, when we will need to emphasize it.

\medskip

\subsection{The Green currents $T^+$ and $T^-$}

Let $f$ be a Hénon map. Denote by $G^\pm$ the
functions $G^\pm (z,w) := \lim_{n\to \infty} d^{-n} \log^+ \|f^{\pm 1}(z,w)\|$, where $\log^+ (\cdot) = \max (0,\log (\cdot))$. Such functions, usually called the \emph{Green functions} of $f$
and $f^{-1}$ respectively,
are well defined (as the convergences are uniform of every compact subset of $\mathbb C^2$).
They are H\"older continuous and plurisubharmonic
\cite{FS92henon,Hubbard86Henon}. 
Hence, the 
$(1,1)$-currents given by
$T^\pm := dd^c G^\pm$ are positive closed.
They
are the \emph{Green currents}
of $f$ and $f^{-1}$
and they 
describe -- in a quantified sense
--
the distribution of the iterate of curves under forward and backward iteration of $f$ \cite{BS91-BS1,FS92henon}.
Their support is equal to $J^\pm:= \partial K^\pm$, respectively. More precisely, they 
are the unique positive $dd^c$-closed currents supported on $K^\pm$, respectively
\cite{DS14rigidity}.
By \cite{BS91-BS2}, the convergence
\begin{equation}\label{eq:conv-henon}
d^{-n} [f^{-n} (M)] \to c_M T^+
 \end{equation}
 holds
for every locally closed submanifold $M\subset \mathbb C^2$ satisying $M\subset J^+$ or $M\subset X$, where $X$ is algebraic. Here $c_M$ is a constant depending on $M$, and we have $c_M >0$ if, for instance, one has $T^-|_{M}>0$.
A similar property holds for $T^-$. 

In the semi-local setting described in Section
\ref{ss:filtration}, 
the convergence above can improved. For a given large $R$, recall that we denote
by $\tilde f$ the associated horizontal-like map on $D_R$. By \cite{Duj04henon}, we can associate to $\tilde f$ its Green current
$\tilde T^+$, which is vertical, and the Green current of
$\tilde f^{-1}$, $\tilde T^-$, which is horizontal in $D_R$.
By \cite{Duj04henon,DNS08dynamics}, we have
\begin{equation}\label{eq:conv-hlm}
d^{-n}  \tilde f^* R \to c_R \tilde T^+
\end{equation}
for every positive closed
vertical current on $D_R$. 
Here $c_R>0$ is a constant depending on $R$
(and it is equal to its \emph{vertical mass}, see \cite{Duj04henon,DS06geometry}; when $R$ is smooth, it is equal to the mass of the restriction of $R$ to any horizontal line in $D_R$).
It follows from \eqref{eq:conv-henon} and \eqref{eq:conv-hlm} that $\tilde T^{\pm} = T^{\pm}_{|D_R}$.

\subsection{The equilibrium measure and Pesin boxes}\label{ss:pesin-boxes}

As the Green functions $G^\pm$ are continuous,
the intersection $\mu := T^+\wedge T^-$ is a
well-defined probability measure, and is the unique measure of maximal entropy of $f$ \cite{BLS93measure,Sibony99}. It satisfies remarkable ergodic
properties \cite{BLS93measure,BLS93periodic,Dinh05decay,BD24clt}. We recall here those that we will need in the sequel.
We denote by $\lambda^+$ and $\lambda^-$ the two Lyapunov exponents of $\mu$.
Recall
\cite{BS98-BS5}
that we have $\lambda^+>0$ and $\lambda^-<0$, hence $\mu$ is hyperbolic in the sense of Pesin theory.

By Oseledec's ergodic theorem, there exist
a full measure subset $\mathscr R$ of the support of $\mu$ 
and two measurable
distributions of $1$-dimensional subspaces
$E^s,E^u
\colon
\mathscr R\to T\mathbb C^2$ with 
$E^s(x),E^u (x)\in T_x\mathbb C^2$
with the property that
for every $x\in \mathscr R$, we have
$E^s(x)\neq E^u(x)$,
$Df (E^{s/u} (x))= E^{s/u} (f (x))$
and
\[
\lim_{n\to \infty} \|Df^n(v)\|= \lambda^+ 
\quad 
\forall v \notin E^s (x)
\quad
\mbox{ and }
\quad
\lim_{n\to \infty} \|Df^{-n}(v)\|= \lambda^- 
\quad
\forall v \notin E^u (x).
\]
The angle between $E^s$ and $E^u$ along an orbit is also controlled. 

Given $r>0$ and $x\in \mathscr R$, 
we denote by 
$B^s_r (x)$ and $B^u_r (x)$
the 1-dimensional affine
discs in $\mathbb C^2$ centred at $x$, of
radius $r$, and whose tangent at $x$ is given by
$E^s(x)$ and $E^u(x)$, respectively.
By Pesin theory, 
for every $x$ there is an $r(x)$ such that
the 
stable and unstable
manifolds $W^s(x)$
and $W^u(x)$ of $x$ are locally 
graphs over $B^s_r(x)$
and $B^u_r(x)$, respectively. We denote by $W^s_r(x)$
and $W^u_r(x)$ these local stable and unstable manifolds, respectively.

Fix $r>0$ and denote by $\mathscr R_r$
the set
of points $x\in \mathscr R$ such that $r(x)\geq r$
\footnote{Further conditions should be imposed on $\mathscr R_r$, see \cite[(4.2), (4.3), and (4.4)]{BLS93measure}.
Since we will not need them and they are always true up to a zero-measure subset, we will not focus on this issue here.}.
Let  $F$ be a compact subset of $\mathscr R_r$, and assume that the diameter of $F$ is $\ll r$.
We denote by
$W^s_r(F)$ and $W^u_r(F)$ the union of the sets $W^s_r(x)$ and $W^u_r(x)$ for $x\in F$.
We call the set $P:= W^s_r(F)\cap W^u_r(F)$
the \emph{Pesin box}
\cite{Pesin77lyapunov}
generated by $F$.
By \cite{BLS93measure}
(see also \cite{Duj04henon}), there exist
a compact set
$P^s$ which is homeomorphic to $W^u_r(x)\cap W^s_r (F)$ for all $x\in F$
and
 a compact set $P^u$ 
 which is homeomorphic to $W^s_r(x)\cap W^u_r (F)$ for all $x\in F$. Then,
 $P$ is homeomorphic to $P^s\times P^u$.
 Moreover, there exists a neighbourhood $N=N(P)$
 of $P$, biholomorphic to a bidisk, such that
 (the image of) every $W_N^s (x):=W^s_r (x)\cap N$ 
 (resp. $W^u_N (x):= W^u_r (x)\cap N$)
 is a vertical (resp. horizontal) graph, and for every
 $x,y\in P$ the unique intersection point between $W^s_N (x)$ and $W^u_N (x)$ is in $P$.

It follows from Pesin theory
\cite{Pesin77lyapunov}
that, up to a negligible set, the support of $\mu$ can be covered by
means of just countably many
Pesin boxes.

\subsection{Semi-parabolic dynamics and horn maps}
\label{ss:prelim-parabolic}

Let us 
assume
from now on  that a Hénon map
$f$ has a semi-parabolic point 
of order $2$
at the origin $\mathbb O$
of $\mathbb C^2$. 
A description 
of the local dynamics of $f$ near the semi-parabolic point is given in \cite{Ueda86local1,Ueda91local2,BSU17semi}.
Following \cite{BSU17semi}, we recall here the definitions and results that we will need in the sequel.
We refer to \cite{lavaurs1989systemes,Shishikura00bifurcation} 
for the
earlier one-dimensional counterparts of these definitions and results.

\medskip

Up to suitable changes of coordinates, we can assume that the local form of $f$
near $\mathbb O$ is given by
\[
f(z,w)
=
(z+  z^2 + O(z^3),
bw + O(zw))
\]
for some $0<|b|<1$
(a more precise development is given in \cite{DL15stability}, but we will not need it here).

We denote by $\mathcal B$ the \emph{parabolic basin} of 
$\mathbb O$, i.e., the open set
of points $x$ such that $f^n (x)\to \mathbb O$
as $n\to \infty$.
Observe that, as $f$ is invertible, $\mathcal B$ is connected. There exists a holomorphic submersion $\phi^\iota\colon \mathcal B \to \mathbb C$, called the (one dimensional)
\emph{incoming Fatou coordinate}, that semi-conjugates the dynamics of $f$
on $\mathcal B$ to a translation by $1$ on $\mathbb C$; i.e., we have 
\begin{equation}\label{e:ingoing}
\phi^\iota (f (x)) = \phi^\iota(x) +1 \quad 
\forall x \in \mathcal B.
\end{equation}

There also 
exists a second (open)
holomorphic map $\phi_2\colon \mathcal B\to \mathbb C$ such that the map
$\Phi = (\phi^\iota, \phi_2)\colon \mathcal B\to \mathbb C^2$ is a biholomorphism which satisfies $\Phi(f(x))=\Phi(x) + (1,0)$ for every $x\in \mathcal B$.
Given any $p \in \bcal$, the fiber $\{ q \in \bcal : \phi^\iota(p)=\phi^\iota(q) \}$ is called the 
\emph{strongly stable manifold} of $p$, denoted by $W^{ss}(p)$. It is an injectively immersed entire curve in $\C^2$, and it is characterized by the following property: 
$$q \in W^{ss}(p) \Longleftrightarrow \limn \frac{1}{n} \log d(f^n(p), f^n(q)) = \log |b|.$$

Let us now consider the set of points converging to $\mathbb O$ under the iteration of $f^{-1}$.
This set is an $f$-invariant 
complex curve $\Sigma\subset \mathbb C^2$, with $\mathbb O$ on its boundary. 
The \emph{Fatou parametrization} of $\Sigma$
is a holomorphic map $\phi^{o}\colon \mathbb C\to \mathbb C^2$ satisfying
\begin{equation}\label{e:outgoing}
f (\psi^o (y))= \psi^o(y+1) 
\quad
\forall y \in \mathbb C.
\end{equation}

\begin{defi}
The map
$\mathcal L_0 := \psi^o \circ \phi^\iota  \colon
\mathcal B \cap \Sigma \to \Sigma$ is the \emph{H\'enon-Lavaurs},
or \emph{transition map}
of $f$
(associated to the semi-parabolic point $\mathbb O$).
The map $H_f:= \phi^\iota \circ \psi^o \colon 
(\psi^o)^{-1} (\mathcal B\cap \Sigma)
\to \mathbb C$ is the \emph{lifted horn map} of $f$
(associated to the semi-parabolic point $\mathbb O$).
\end{defi}

The following 
properties of $H_f$ directly follow from its 
definition and the local description of $\mathcal B$ and $\Sigma$,
see \cite{BSU17semi,DL15stability}.
The last item is a
key point in the characterization of bifurcations in \cite{DL15stability} by means of homoclinic tangencies.

\begin{prop}\label{p:properties-H}
The following properties hold.

\begin{enumerate}	
\item The domain $\psi^o(\mathcal B\cap \Sigma)$
contains the set 
$\{|\Im z| >R\}$ for every $R$
large enough.

\item For every $z \in (\psi^o)^{-1}(\mathcal B)$ and 
 every $w \in \C$, we have   $H_f(z)=w$ if and only if $\psi^o(z)$ lies in the intersection of $\Sigma$ and the strongly stable manifold
	$\{(x,y) \in \mathcal B : \phi^\iota(x,y)=w\}$.

\item 
Given $z\in \mathbb C$,
	we have $H'_f(z)=0$ if and only if the  strongly stable manifold
	$\{(x,y) \in \mathcal B : \phi^\iota(x,y)=w\}$ is tangent to $\Sigma$ at $\psi^o(x)$.
\end{enumerate}
\end{prop}

The maps $\mathcal L_0$ and $H_f$
are conjugated to each other, so their dynamics 
are very similar. In particular, one can use the map $H_f:\mathbb C\to \mathbb C$
as a model for the map $\mathcal L_0$,
whose domain and image are in $\C^2$. 
It follows from \eqref{e:ingoing} and \eqref{e:outgoing}
that we have
\[
f\circ \mathcal L_0 = \mathcal L_0 \circ f \quad \mbox{ and }
\quad 
H_f \circ \tau_1 = \tau_1 \circ H_f,
\]
where we denote by $\tau_1$
 the translation by $1$ in $\mathbb C$.
In particular, 
$H_f$
induces
a map $h_f$ on (a subset of) the cylinder 
$\mathcal C$
obtained
by taking the quotient of $\mathbb C$ by the $\mathbb Z$-action of $\tau_1$.
By Proposition \ref{p:properties-H} (1), the domain 
of $h_f$ contains the two 
extremities of  $\mathcal C$.
By \cite{BSU17semi},
$h_f$
 extends to such extremities, that we can identify with $0$ and $\infty$ in $\rs$.
 In particular, we can see $h_f$ 
 as a map from an open subset of $\rs$ (containing two neighbourhoods of $0$ and $\infty$)
 to $\rs$.

 \begin{defi}\label{d:hf}
 The map $h_f$ is the \emph{horn map}
 of $f$ (associated to the semi-parabolic point $\mathbb O$).
 \end{defi}

The maps $\mathcal L_0$, $H_f$,
and $h_f$
are
deeply related to the so-called
(semi-)parabolic implosion phenomenon for the perturbations of the map $f$, see \cite{BSU17semi}
and \cite{lavaurs1989systemes,Shishikura00bifurcation} 
for their counterparts in one-dimensional parabolic dynamics. 
While we will not need
results in this direction in the proof 
of Theorem \ref{th:islandp}, 
we recall here what we will need in
the proof of Theorem \ref{th:maxdim}.

\medskip

Observe that
the map $\mathcal L_0$
is actually defined on $\mathcal B$. 
For every $\alpha\in \mathbb C$, define the \emph{Hénon-Lavaurs maps of phase $\alpha$}
$\mathcal L_\alpha\colon \mathcal B\to \mathbb C^2$ as
\[
\mathcal L_\alpha := \psi^o \circ \tau_\alpha\circ \phi^\iota,
\]
where $\tau_\alpha$ denotes the translation by $\alpha\in \mathbb C$ in $\mathbb C$.
Observe that the image of $\mathcal L_\alpha$ is contained in $\Sigma$, and that 
the definition of $\mathcal L_0$ is coherent with 
its previous definition above.

\medskip

For small $\epsilon$, 
we consider 
holomorphic perturbations $f_\epsilon$
of $f$ of the form
\[
f_\epsilon (z,w)
=
(z+z^2 + \epsilon^2 + O(z^2),
b_\epsilon w + O(zw)),
\]
where in particular
$b_\epsilon$ depends holomorphically in $\epsilon$.
Following \cite{BSU17semi}, we say that a sequence 
$(n_j, \epsilon_j)\subset (\mathbb N, \mathbb R^+)^{\mathbb N}$ is an 
\emph{$\alpha$-sequence} if $\epsilon_j\to 0$
and $n_j-\pi/\epsilon_j\to \alpha$. Observe that this condition implies that $n_j\to \infty$, and prescribes that the
convergence $\epsilon_j\to 0$ happens tangentially to the 
positive
real axis.

\begin{thm}
[Bedford-Smillie-Ueda \cite{BSU17semi}]\label{t:BSU}
    Let $(n_j,\epsilon_j)$
    be an \emph{$\alpha$-sequence}. Then, 
    \[
f^{n_j}_{\epsilon_j}\to \mathcal L_\alpha    
    \]
    locally uniformly in $\mathcal B$.
\end{thm}

\section{Proof of Theorem \ref{th:islandp}}

We continue to use the notation of the previous sections.
Let $R>0$ be large enough 
so that the filtration property
in Section \ref{ss:filtration}
holds for $f$ and $R$.
We will only be 
concerned with the dynamics of $f$ inside $D_R$, where we recall that $f$ can be seen as an invertible horizontal-like map.
Fix 
$z_0 \in \C$ and
denote by $S(z_0)$ 
a 
connected component of $D_R \cap \{\phi^\iota(x,y)=z_0\}$.
We also let $D$ be a small disk in $\Sigma$ intersecting the boundary of $\mathcal B$. As we will see, the proof 
of Theorem
\ref{th:islandp} will essentially consist
in finding suitable
intersections
between preimages of $S(z_0)$ and images of $D$.

\medskip

In order to find such intersections, 
let us fix  
an arbitrary Pesin box $P$ for $\mu$,
see Section \ref{ss:pesin-boxes}.
 Recall that we denote by
 $N$
 a given neighbourhood
of $P$, where the stable/unstable manifolds associated to $P$ can be thought as vertical/horizontal, see Section \ref{ss:pesin-boxes}.

\medskip

We first consider the preimages of $S(z_0)$.
Here is where the semi-local 
setting of horizontal-like maps will turn out to be useful, as $S(z_0)$ does not a priori satisfy the conditions for the convergence \ref{eq:conv-henon}.

\begin{lem}\label{lem:conv-vertical-T+}
	We have $d^{-n} (f^n)^* [S(z_0)] \to T^+$
 on $D_R$.
\end{lem}

\begin{proof}
Observe that,
up to choosing $R>0$ large enough,
we have $S(z_0) \subset K^+$,
which implies that
$S(z_0) \cap V_R^+=\emptyset$. Hence,
$S(z_0)$
is a vertical analytic set in $D_R$.
Moreover, we have $[S(z_0)]\wedge [L]=1$ for every horizontal line $L$ in $D_R$.
	The assertion is then 
 a consequence of \eqref{eq:conv-hlm}.
\end{proof}

\begin{lem}\label{lem:intersectP}
	For $\mu$-a.e.\ $x_1 \in P$, there exists $n_1 \in \N$ such that $f^{-n_1}(S(z_0))$ intersects $W^u(x_1)$  transversally  at some $y_1 \in P$.
\end{lem}

\begin{proof}
In the case where $S(z_0)$ is 	
replaced by a disc $\Delta$
in the stable manifold of a saddle point, the assertion follows from \cite[Lemma 9.1]{BLS93measure}.
Since the proof only uses that $\Delta$ satisfies $d^{-n} [f^{-n} \Delta]\to T^+$, the same proof applies here thanks to Lemma \ref{lem:conv-vertical-T+}.
\end{proof}

For any $r>0$, let $\mathcal S(r)$ denote the connected component of 
$\{(x,y) \in D_R:  \phi^\iota(x,y) \in \D(z_0,r)  \}$
which contains $S(z_0)$.
For any $z \in \D(z_0,r)$, we let $S(z)$ denote the connected component of $\{\phi^\iota(x,y)=z\}$ contained in $\mathcal S(r)$. 

\medskip

By Lemma \ref{lem:intersectP}, there exist
$n_1 \in \N$ and $x_1 \in P$ such that $f^{-n_1}(S(z_0))$ intersects transversally $W^u_N(x_1)$ in $P$. By the description of the Pesin boxes in Section \ref{ss:pesin-boxes}
and continuity, we get the following result.

\begin{lem}
\label{def:r0}
For every $z_0\in \mathbb C$
there exists
$r(z_0)>0$ such  that 
	\begin{enumerate}
		\item the connected component $U$ of $\mathcal S (r(z_0)) \cap W^u_N (x_1)$ containing $y_1$
  is simply connected and relatively compact in $N$; and
		\item	for all $z \in \D(z_0,r(z_0))$,
	$f^{-n_1}(S(z))$  intersects  $W^u_N(x_1)$  transversally at a point in $U$.
	\end{enumerate}
\end{lem}

We now prove a result analogous to Lemma \ref{lem:intersectP} 
for a disk in $\Sigma$ intersecting the boundary of $\mathcal B$.
We first need the following analogous of Lemma \ref{lem:conv-vertical-T+}.

\begin{lem}\label{lem:conv-t-}
	Let $D \subset \Sigma$ be a small disk which intersects $\partial \mathcal B$.
	Then
 \begin{enumerate}
 \item ${T^{+}}_{|D}>0$, and
 \item $d^{-n} (f^n)_*[D] \to c_D T^-$, for some positive constant $c_D$.
 \end{enumerate}
\end{lem}

\begin{proof}
	By assumption, there exists $x \in \partial \mathcal B \cap D$, so that $G^+(x)=0$.
	Let us prove that there exists $x_1 \in D$ such that $G^+(x_1)>0$.
	Assume, by contradiction, that $G^+\equiv 0$ on $D$. Then,
 we have 
 $D \subset K^+ \cap K^-$. In particular, we have
  $f^n(D) \subset D_R$
for all $n \in \N$. 
Therefore, $\{f^n : D \to \C^2\}$ is a normal family. 
As there is an open subset of $D$ contained in $\mathcal B$, any limit of
any subsequence
$f_{|D}^{n_k}$ must be constantly equal to the semi-parabolic point. This implies that
$D  \subset \mathcal B$ and gives a contradiction with the assumption that $D\cap \partial \mathcal B \neq \emptyset$.

	Therefore, $G^+ : D \to \R$ is not constantly equal to $0$. On the other hand, it is a non-constant subharmonic function on $D$
	which admits a local minimum at $x \in D$. It follows  by the maximum principle that $G^+$
  cannot be harmonic on $D$, which
	 means that $dd^c G^+>0$ on $D$. This proves the first assertion.

By assumption, we have
$D \subset \Sigma \subset J^-$. Applying
 \eqref{eq:conv-henon}
(with $f^{-1}$ instead of $f$),
by the first item
 we have  $d^{-n} (f^n)_*[D] \to c_D T^-$, for some $c_D>0$. The assertion follows.
\end{proof}

\begin{lem}\label{lem:intersectionP-D}
	For $\mu$-a.e.\ $x_2 \in P$, there exists $n_2 \in \N$ such that $f^{n_2}(D)$ intersects transversally $W^s_N(x_2)$ at some $y_2$ in $N$.
\end{lem}

\begin{proof}
	As in Lemma \ref{lem:intersectP}, 
 thanks to Lemma \ref{lem:conv-t-}
the proof is the same as that of \cite[Lemma 9.1]{BLS93measure}.
\end{proof}

The following lemma gives the desired 
tranverse 
intersections between preimages of disks in stable manifolds in $\mathcal B$ and images of disks in $\Sigma$ intersecting the boundary of $\mathcal B$.

\begin{lem}\label{lem:inclination}
	There are
 two
 sequences of integers $N_j, M_j \to +\infty$ such that the following properties hold.
 \begin{enumerate}
	\item $f^{-N_j}(U) \subset N$ for all $j$;
 \item 
$\mathrm{diam}(f^{-N_j}(U)) \to 0$ as $j \to +\infty$;
 \item 
	for all $z \in \D(z_0,r(z_0))$,
 the connected component of $f^{-n_1-N_j}(S(z))\cap N$ intersecting $f^{-N_j}(U)$
 is a vertical graph in $N$;
 \item $f^{M_j} (y_2)\in N$ for all $j$;
 \item 
the connected component of  $f^{n_2+M_j}(D) \cap N$
containing $f^{M_j} (y_2)$
is a horizontal graph in $N$.
\end{enumerate}
In particular, 
for all $z \in \D(z_0,r(z_0))$, $f^{-n_1-N_j}(S(z))$
intersects 
 $f^{n_2+M_j}(D)$ 
 transversally at a point in $N$.
\end{lem}

\begin{proof}
 We fix points 
 $x_1, x_2$ and integers  $n_1, n_2$ as in Lemmas \ref{lem:intersectP} and \ref{lem:intersectionP-D}.
 
	By Poincaré's recurrence theorem,  
	there exists $N_j \to \infty$ such that $f^{-N_j}(x_1) \in P$.
	Since $U \subset W^u_N(x_1)$, we have $f^{-N_j}(U) \subset N$.
 Up to extracting a subsequence, we may assume that $\lim_{j \to +\infty} f^{-N_j}(x_1) \to \hat{x}_1 \in P$, and 
 that
 we have $f_{|U}^{-N_j} \to \hat{x}_1$ uniformly. This proves (1) and (2).

	Let $A_j$ denote the connected component of $f^{-N_j-n_1}(\mathcal S(r(z_0))) \cap N$ containing $U$; by definition,
 $A_j$ is a union of 
	pieces  of strongly stable manifolds $\{\phi^\iota = z-n_1-N_j\}$. 
 Let $\tilde S$ be one such piece. 
 By the inclination lemma, each $\tilde S$ converges 
	in the $C^1$ topology to $W^s_N(\hat {x}_1)$, 
 which by Pesin theory (see the discussion in Section \ref{ss:pesin-boxes}) is a horizontal graph in $N$. This proves (3).

	On the other hand, by a symmetric argument, there exists $M_j \to +\infty$ such that $f^{M_j}(x_2) \in P$, and $f^{M_j}(x_2) \to \hat{x}_2 \in P$.
 Since $f^{n_2}(D)$ intersects $W^s(x_2)$ transversally in $P$, 
	the inclination lemma implies that the connected component of 
	$f^{M_j + n_2}(D) \cap N$ containing $f^{M_j}(y_2)$ converges in the $C^1$ topology
	to $W^u_N(\hat{x}_2)$, which is a horizontal graph in $N$. 
    This proves (4) and (5).
\end{proof}

\begin{proof}[End of the proof of Theorem \ref{th:islandp}]
 Set
 $W:=(\psi^o)^{-1}(\mathcal B)$.
By the definition of $h_f$ and $H_f$, 
 and in particular by the fact that
$H_f$
commutes with the translation  $\tau_1$,
 it is enough to prove the following property.

\medskip

 \begin{itemize}
\item[\textbf{($\star$)}]
for every $z_0\in \C$
there exists $r(z_0)>0$
such that
 for any open set $\tilde W$ intersecting $\partial W$
there exist 
 $n \in \Z$
	and $V \Subset W \cap (\tilde W+n)$ such that $H_f: V \to \D(z_0, r(z_0))$ is a conformal isomorphism.
 \end{itemize}

 \medskip

Fix $z_0 \in \C$ and, without loss of generality,
 let $\tilde W$ be an open ball 
 intersecting $\partial W$. 
  Set 
 $D:=\psi^o(\tilde W) \subset \Sigma$, and observe that $D$ satisfies the assumption of Lemma \ref{lem:conv-t-}.
	Let also $r(z_0)$ and $U$ be as in Lemma \ref{def:r0},
 $n_1,n_2$ be as in Lemmas \ref{lem:intersectP} and \ref{lem:intersectionP-D},
 and let the sequences 
 $\{N_j\}, \{M_j\}$
 be  as in Lemma \ref{lem:inclination}. 
 By Lemma \ref{lem:inclination},
for every $z \in \D(z_0,r(z_0))$ and every $j$, 
	the set $f^{-n_1- N_j-n_2-M_j}
 (S(z))$ intersects $D$ transversally at a point in
 $f^{-n_2-M_j}(U)$.
By the second item in 
Lemma \ref{lem:inclination}, we can fix $j_0$ such that the diameter of $f^{-n_1-N_j}(U)$ is much smaller than the diameter of $N$.
For convenience, we also set $\tilde n := n_1+n_2 + N_{j_0}+ M_{j_0}$.
\medskip

Set $V:=
\tilde n + (\psi^o)^{-1} (f^{-n_2-M_{j_0} -N_{j_0}} (U)) \subset (\psi^{o})^{-1} (D) = \tilde n + \tilde W$.
By the choice of $j_0$ and Lemma \ref{lem:inclination} (4), we have
 $V\Subset \tilde n+ \tilde W$.
For every $v\in  V$, let $u=u(v)$ be the point in $U$ given by
$u=f^{n_2+M_{j_0} + N_{j_0}} (\psi^o(v - \tilde n))$.
The map $v\mapsto u(v)$ is a biholomorphism from $V$ to $U$.
Recall also that, by
construction, the map $\phi^\iota \circ f^{n_1}$ is a biholomorphism between $U$ and $\mathbb D(z_0,r(z_0))$.
Hence, the map
$\tilde H (\cdot)
:=  \phi^\iota \circ f^{\tilde n }
\circ \psi^o (\cdot - \tilde n)$
is a biholomorphism from $V\Subset\tilde n+ \tilde W$ to $\mathbb D(z_0, r(z_0))$.
Recalling that $f\circ \psi^o (\cdot)  = \psi^o (\cdot +1)$,
we see that $\tilde H = H_f$.
Hence, $H_f$
is a biholomorphism from $\tilde n + V \Subset \tilde n + \tilde W$ to $\mathbb D(z,r(z_0))$. This
proves {\bf ($\star$)}, and the assertion follows.
\end{proof}

\begin{rem}
Observe
that our proof of Theorem \ref{th:islandp} does not really need that $f$ is a globally defined  Hénon map, but it is enough that it is an invertible horizontal-like map. In particular, the map $f$
does not need to be algebraic.
\end{rem}






\section{Almost maximal dimension for $J^+$ and Theorem \ref{th:maxdim}}

\subsection{Preliminaries and McMullen's result}

In order to prove Theorem \ref{th:maxdim}, the idea will be to first construct a hyperbolic (repelling) set of large Hausdorff dimension for some horn map, and then to transfer it 
 to some suitable perturbations of the initial semi-parabolic Hénon map.
 While we could follow more closely 
 Shishikura's arguments \cite{Shishikura98hausdorff}, 
 we will use a result by McMullen
\cite{mcmullen}, see Theorem \ref{th:mcmu} below,
to prove that there are small quasiconformal copies of the Mandelbrot set in the parameter space of horn maps. This will allow us to
bypass a part of the proof.
In this preliminary section, 
we recall some terminology and McMullen's
result.

\begin{defi}
    Let 
    $W\subset \rs$ be an open set and  $f: W \to \rs$ a holomorphic map. We say that $z \in \rs$
 is \emph{unramified}     (for $f$)
    if the closure of the set $\{ x \in W : \exists n \in \N, f^n(x)=z \text{ and } (f^n)'(x) \neq 0  \}$ is dense in 
    $J_F(f)$.
\end{defi}

For rational maps, some points may be completely ramified. For instance, it may happen that every 
preimage of some point is a critical point. However, we will see
in Lemma \ref{l:unramified}
 that this is not possible for maps with the small island property.

\begin{defi}\label{d:local-misiurewicz}
Let 
$M$ be a complex manifold
and $(f_\la)_{\lam \in M}: W \to \rs$
a holomorphic family of holomorphic maps.
    We
    say that the family
    $(f_\la)_{\lam \in M}$
    has a \emph{local Misiurewicz bifurcation of degree $d \geq 2$ at $\lam_0 \in M$} 
    if:
    \begin{enumerate}
        \item there exists a  marked critical point $c_\lam$ of $f_\la$
        and $n\in \N$ such that $x_{\lam_0}:=f_{\lam_0}^n(c_{\lam_0})$ is a repelling periodic point;
        \item the map $\lambda \mapsto f_\lam^n(c_\lam)$ is not constantly equal to $x_\lam$ in a neighborhood of $\lam_0$;
        \item $c_{\lambda_0}$ is unramified;
        \item the local degree of $f_\la^n$ at $c_\lam$ is constantly  equal to $d$ in a neighborhood of $\lam_0$.
    \end{enumerate}
\end{defi}

Recall that a critical point $c_\lambda$
is \emph{marked}
if it can be followed holomorphically as a function $\lambda \mapsto c_\lam =c(\lambda)$
of the parameter $\lambda$. Such a
critical point is \emph{active} at $\lambda_0$ if the sequence
$f^n_\lambda (c_\lambda)$ is not normal on any neighbourhood of $\lambda_0$.
As in \cite{mcmullen}, we remark that the subtle point in the above definition 
is the fourth request.

\medskip

We can now recall the main technical results in \cite{mcmullen}.
We will not use the following 
proposition, but we quote it since proving a version of it in our context will be a main point of our construction.

\begin{prop}\label{p:mcmullen-misiurewicz}
Let 
$M$ be a complex manifold
and $(f_\la)_{\lam \in M}: \rs \to \rs$
a holomorphic family of rational maps.
Assume that 
a marked critical point $c_\lambda$
is active and unramified at $\lam_0$.
Then, there exists 
two sequences $M\ni \lambda_n \to \lambda_0$ and
$\N \ni m_n \to \infty$
such that, for every $n\in \N$,
the family $f^{m_n}$ has a local Misiurewicz bifurcation at $\lambda_n$.
\end{prop}

We denote by $\mathbf{M_d}$ the \emph{degree $d$ Mandelbrot set}, i.e., the connectedness locus of the family $z^d + c$, for $c\in \C$.

The theorem below was proved by McMullen in the context of holomorphic families of rational maps.
However, the arguments are purely local (indeed, the proof consists in finding suitable polynomial-like restrictions
of the maps $f_\la$) and the proof carries through verbatim in our more general setting.

\begin{thm}[McMullen \cite{mcmullen}]\label{th:mcmu}
    If a
    holomorphic family
    $(f_\lam)_{\lam \in M} \colon W\to \rs$ has a local Misiurewicz bifurcation of degree $d$ at some parameter $\lam_0$, then there is a sequence of embeddings $\phi_n \colon \mathbf{M_d} \to M$
    such that for every $c \in \mathbf{M_d}$, we have
    $\lam_n:=\phi_n(c) \to \lam_0$ and,
    for every $n\in \N$,
    $f_{\lam_n}$ has a polynomial-like restriction which is hybrid-equivalent to $z^d+c$.
    Moreover, the quasiconformal distortion of these copies tends to $0$ as $n \to +\infty$.
\end{thm}

In practice, McMullen's
results 
give that, whenever
-- in a family of rational maps 
--
a bifurcation occurs, one can create 
local Misiurewicz bifurcations, and thanks to them, by means of a renormalization process,
a large set of bifurcation parameters. As a consequence, by \cite{Shishikura98hausdorff}, one can then find 
maps with large hyperbolic dimension close to any bifurcation parameter. 
In the next section, we will adapt the above ideas
in the setting of maps with the small island property.

\subsection{Large hyperbolic sets from the small island property}

In this section,
we fix an open set $W\subset \rs$
and a holomorphic map $h\colon W\to \rs$ 
satisfying the small island property.
We also assume that $\partial W \neq \emptyset$ and 
that $h$ has at least one critical point (both of these properties hold for horn maps of dissipative semi-parabolic Hénon maps; in particular, the existence of a critical point was proved in \cite{DL15stability}).
For every $\lam \in \C^*$, we denote
$h_\lambda := \lambda h$.
The following is the main result of this section.

\begin{prop}\label{p:big-hyperbolic-island}
For every $\eps>0$  
there exists $N \in \N$,  $\lam_1 \in \C^*$,
simply connected domains $U, U_1, \ldots, U_N \subset W \cap \C^*$, 
and integers $m_1, \dots, m_N$
such that:
\begin{enumerate}
    \item for all $1 \leq i <j\leq N$,
    we have $U_i,U_j \Subset U$ and $\overline{U_i} \cap \overline{U_j} = \emptyset$;
    \item for all $1 \leq i \leq N$,
    the map
    $h_{\la_1}^{m_i}: U_i \to U$ is a conformal isomorphism; 
    \item if $c:=\max_{1 \leq i \leq N} \sup_{z \in U_i} |(h_{\la_1}^{m_i})'(z)|$, then 
    $$2-\eps \leq \frac{\log N}{\log c}.$$
\end{enumerate}
\end{prop}

In particular, it follows from the Bowen formula \cite{Bowen75notes} that 
the limit set of the Conformal Iterated Function System (CIFS) given by 
the holomorphic maps
$h_{\la_1}^{-m_i}: U \to U_i$
as in the above statement
has Hausdorff dimension at least $2-\eps$.
\medskip

We begin with the following preliminary result, which we will use to replace the non-normality of the critical orbit near a bifurcation parameter in a family of rational maps.

\begin{lem}[Shooting Lemma]\label{lem:shooting}
Let $(h_\lam)_{\lam \in \C^*}$ be the holomorphic family defined by
$h_\lam = \lam h$ for every $\lam \in \C^*$.
Fix $\la_0\in \C^*$
and let
$\gamma_1, \gamma_2$ be two holomorphic maps defined in a neighborhood of $\lam_0$ and such that 
    $h^n_{\lam_0} \circ \gamma_1(\lam_0) \in \partial W$ for some $n \in \N$.
	Then there exists $\la'$ arbitrarily close to $\la_0$  such that
	$$h_{\lam'}^{n+1}(\gamma_1(\lam')) = \gamma_2(\lam').$$
\end{lem}
In the proof of
Lemma~\ref{lem:shooting} we will need the following consequence of the argument principle.

\begin{lem}
\label{lem:fp} 
Let $V$ be a Jordan domain, and let $f,g$ be holomorphic functions in a neighborhood of $\overline{V}$. Suppose that  $g(\overline{V})\subset f(V)$ and  $g(\partial V)\cap f(\partial V)=\emptyset$. Then there exists $\lambda\in V$ such that $f(\lambda)=g(\lambda)$.
\end{lem}

\begin{proof}[Proof of Proposition~\ref{lem:shooting}]
For every $\lam\in \mathbb C^*$, 
set $G(\lam):=h_\lam^n(\gamma_1(\lam))$ and $g(\lam):= \lam^{-1} \gamma_2(\la)$. Then, the equation 
    $h_\lam^{n+1}(\gamma_1(\lam))=\gamma_2(\lam)$
    may be rewritten as
    $$h_\lam \circ G(\lam)=g(\lam).$$
    We will apply Lemma \ref{lem:fp} to 
    the functions $f (\lambda):= h_\lam \circ G (\lambda)$ and $g (\lambda)$. 

Let $D$ be a disk centered at $g(\lam_0)$ small enough for the small island property of $h$ to apply
    (i.e., whose radius is smaller than $r(g(\lambda_0))$).
    Fix also an $\eps>0$ very small compared to the diameter of $D$, and also 
    small enough so that 
    $G: \D(\lam_0,\eps) \to G(\D(\lam_0,\eps))$ is a branched cover without critical points besides possibly $\lam_0$.
    By assumption, $G(\D(\lam_0,\eps))$ intersects $\partial W$.
	
	By the small island property of $h$
 (which implies the small island property of  $h_\lambda$ for every $\lambda \in \mathbb C^*$, see Remark \ref{r:natural}), 
 there exists $U \Subset G(\D(\lam_0,\eps))$ such that $f_{\la_0}(U)=D$. 
	Let $V$ denote a connected component of $G^{-1}(U)$ inside $\D(\la_0,\delta)$. By the choice of $\eps$, $V$ is a Jordan domain. By construction, we have $f \circ G(U)=D$. On the other hand, 
    $g(U)$ is contained in the disk $\D( g(\lam_0), C \eps^{1/d})$, where $d$ is the local degree of $G$ at $\lam_0$ and $C$ is a positive constant independent of $\epsilon$. 
    Therefore, for $\eps$ small enough, 
    we have $g(U) \Subset f \circ G(U)$ and $g(\partial U) \cap f \circ G(\partial U) = \emptyset$. It follows from 
    Lemma \ref{lem:fp}
    that there exists $\lam' \in U$ such that $f \circ G(\lam')=g(\lam')$. The proof is complete.
\end{proof}

Proposition \ref{p:mcmullen-misiurewicz} has an assumption on the non-ramification of the critical point. In our setting, we will be able to get rid of that assumption thanks to the following lemma.

\begin{lem}\label{l:unramified}
  Every 
    $z \in \C^*$ is unramified for $h$.
\end{lem}

\begin{proof}
  Fix $z \in \C^*$ and let $U \Subset W$ be an open set intersecting $J_F(h)$. Define
    \[B^-(z):=\{x \in W : \exists n \in \N, h^n(x)=z \text{ and } (h^n)'(x) \neq 0 \}.\]
    It is enough to show that
    $U \cap B^-(z) \neq \emptyset$.

    As $U\cap J_F(h)\neq \emptyset$,
    by Montel's lemma
    there exists $n \in \N^*$ such that  $h^n(U) \cap \partial W \neq \emptyset$. 
    We choose the smallest such $n$.
    The map $h^n: U \to h^n(U)$ has only finitely many critical points. On the other hand, 
    by the small island property, there are infinitely many $y \in h^n(U)$ such that $h(y)=z$ and $h'(y) \neq 0$. 
    In particular, we can find one such $y$ such that there exists $x \in U$ with $h^n(x)=y$
    and $(h^n)'(x) \neq 0$. As such $x$ belongs 
    to $U\cap B^-(z)$, this concludes the proof.
\end{proof}

We can now prove the following version of Proposition \ref{p:mcmullen-misiurewicz}.

\begin{prop}\label{prop:existencem}
Let $W\subset \rs$ be an open set with $\partial W \neq \emptyset$ and $h:W\to\rs$  
a holomorphic map satisfying the small island property
and with at least one critical point.
   There exists $\lam_0 \in \C^*$ such that the
    family
    $(h_\lam:=\lam h)_{\lam \in \C^*}$
    has a local Misiurewicz bifurcation at $\lam_0$.
\end{prop}

\begin{proof}
Take $\lambda_0\in \mathbb C^*$. We allow ourselves
to modify $\lambda_0$ in what follows.
We let $\gamma_1 (\lambda)$ be the motion of a critical point near $\lam_0$, and $\gamma_2 (\lambda)$ 
the motion of a repelling periodic cycle
(which exists by Theorem \ref{th:islandp} and the implicit function theorem).
Observe that, as
the maps $h_\la$ have the form $h_\la = \la h$, the map $\gamma_1$ is in fact constant;
we will therefore simply write $\gamma_1$ instead of $\gamma_1(\la)$.
As we are allowed to modify $\lam_0$, we can also assume that $\lam_0 h(\gamma_1) \in \partial W$.
Hence, we are in the assumptions of Lemma \ref{lem:shooting}.
That lemma shows that, up to slighly modifying $\lam_0$, we have a Misiurewicz relation at $\lam_0$.
As in
\cite{mcmullen}, since we are allowed to slightly perturb the starting parameter (and the critical point is automatically unramified by Lemma \ref{l:unramified}),
the first three conditions in Definition \ref{d:local-misiurewicz}
can be achieved from the above construction.
Hence, we
only have to show that
(up to a further small perturbation)
we can get the fourth condition in Definition \ref{d:local-misiurewicz}.
Let us recall that, as in \cite{mcmullen}, this property may fail if the local degree
of the critical point at $\lam_0$ is larger than
at nearby parameters $\lam'\neq \lam_0$.

We assume for simplicity that the integer $n$
in Definition \ref{d:local-misiurewicz} is equal to $1$.
Following \cite{mcmullen}, we denote by $a_\lam$ the 
holomorphic 
motion of the repelling point giving the Misiurewicz relation at $\lam_0$. 
We denote by $U$ a small linearization domain for $a_{\lam}$, for all
$\lam$ in 
given small
neighbourhood
of $\lam_0$. We will always work with $\lam$ in this neighbourhood. We can also assume that the 
local degree of $c_\lambda$ is constant outside of $\lambda_0$.
We let $b_\lambda$
be the motion of a second repelling point, which stays in $U$ for all $\lambda$ in consideration.
By the small island property, there exist preimages
of $b_\lambda$ accumulating on
$a_\lambda$. We denote by $b'_\lambda$ one of these preimages.
Applying again Lemma \ref{lem:shooting}, we can find $\lam'$
close to $\lam_0$
such that $f(c_{\lambda'})= b'_{\lambda'}$.
The local degree of $f$ is constant near $\lambda'$, and since there are no critical points in $U$, we see that the same is true for the iterate of $f$
 mapping $c_{\lam'}$ to $b'_{\lambda'}$.
The proof is complete.
\end{proof}

\begin{proof}[Proof of Proposition \ref{p:big-hyperbolic-island}]
Fix $\eps>0$ and let $d \geq 2$ be the local degree at a given critical point of $h$.
For any $c \in \C$, denote
    $g_c(z):=z^d+c$. 
    By \cite{Shishikura98hausdorff}, there exists $c \in \mathbf{M_d}$,  $N \in \N$, 
    simply connected domains $V, V_1, \ldots, V_N \subset \C$,
    and integers $m_1, \dots, m_N$
    such that:
\begin{enumerate}
    \item for all $1 \leq i <j\leq N$,
    we have $V_i, V_j \Subset V$ and $\overline{V_i} \cap \overline{V_j} = \emptyset$;
    \item for all $1 \leq i \leq N$, 
    $g_{c}^{m_i}: V_i \to V$ is a conformal isomorphism;
    \item if $A:=\max_{1 \leq i \leq N} \sup_{z \in V_i} |(g_{c}^{m_i})'(z)|$, then 
    $$2-\frac{\eps}{2} \leq \frac{\log N}{\log A}.$$
\end{enumerate}

    By Proposition \ref{prop:existencem}, there exists a local Misiurewicz bifurcation of degree $d \geq 2$
    at some $\lam_0\in \C^*$ in the family $(h_\la:=\la h)_{\la \in \C^*}$.
    By Theorem \ref{th:mcmu}, there exists a sequence $\lam_n \in \C^*$ and, for every $n$, a $K_n$-quasiconformal homeomorphisms $\phi_n: V \to \phi_n(V)$ with
    $K_n \to 1$ as $n\to \infty$
    which conjugates $g_c$ to a polynomial-like restriction of $h_{\la_n}$. Then,
    for $n$ large enough, the open sets  $U:=\phi_n(V)$ and $U_i:=\phi_n(V_i)$
    satisfy the 
    conditions in the statement.
\end{proof}

\subsection{Proof of Theorem \ref{th:maxdim}}

By Theorem \ref{th:islandp}, we can apply Proposition \ref{p:big-hyperbolic-island} 
to the 
horn map $h_f$ 
associated to the semi-parabolic fixed point of
$f$
as in Definition \ref{d:hf}. 
The following proposition is a rewriting of that statement in terms of the 
Hénon-Lavaurs maps $\mathcal L_\alpha$. Observe that the multiplicative constant in Proposition \ref{p:big-hyperbolic-island}
translates to the phase $\alpha$ in the statement below.

\begin{prop}\label{p:ifs}
For every $\epsilon >0$
there exists $\alpha \in \C$,
disjoint open sets $U,U_1, \ldots, U_N \subset \Sigma$,
and integers $m_1, \ldots, m_N$ 
    such that
    \begin{enumerate}
    \item 
    for every
    $1\leq i<j \leq N$, we have $U_i,U_j\Subset U$ 
    (in the topology induced by $\Sigma$)
    and $\overline{U_i} \cap \overline{U_j} =\emptyset$;
    \item
     for every
    $1\leq i \leq N$,
    the map
    $\lcal_\alpha^{m_i}: U_i \to U$ is a conformal isomorphism;
\item if $c:=\max_{1 \leq i \leq N} \sup_{z \in U_i} |(\mathcal L_{\alpha}^{m_i})'(z)|$, then 
    \begin{equation}\label{e:Nc}
    2-\eps \leq \frac{\log N}{\log c}.
    \end{equation}
    \end{enumerate}
\end{prop}

In the proposition above, up to taking preimages,
we may assume that $U$ and all
 the $U_i$'s have 
 small diameter, and that they
 are transverse to the strong stable foliation in $\mathcal B$.
  In particular, we may choose a
  small open set $V\subset \C^2$ with $V\cap \Sigma =U$ and a coordinate system $(x,y) \in \D^2$ for $V$
  in which $U=\D \times \{0\} $ and the strongly stable foliation is the vertical foliation
  given by $x=c$, for $c\in \D$.
  We allow ourselves to reduce the size of $V$
  (in the transversal direction to $\Sigma$)
  in the following, as well as $U$ and the $U_i$'s. We only work on $V$
  (resp. $\D^2$)
  in the following. By a slight abuse of notation, we will still denote by $U_i\times \{0\}$
  the images of the $U_i$ in the chart $(x,y)$. Moreover,
  given two vertical subsets $A,B\subset \D\times \D$,
  we will write $A\Subset B$
  whenever $A\cap \{y=y_0\} \Subset B \cap \{y=y_0\}$ 
  for every $y_0 \in \D$. 
  For every $1\leq i \leq N$,
  we also denote $V_i := U_i \times \D$ and by
  $L_i$ the reading of $\mathcal L_\alpha^{m_i}$
  in the coordinates $(x,y)\in \D^2$. 
  We also 
  denote by
  $L$ the holomorphic map
  on $\cup V_i$
  which is equal to $L_i$ on $V_i$.
  Observe that, in particular, 
  with the above notations 
  we have 
  \begin{equation}
  \label{e:incl-L}
  L_i^{-1} (V_j) \Subset V_i
  \mbox{ for all } i,j  \quad 
  \mbox{ and }
  \quad L^{-1} (\cup V_i) \Subset \cup V_i.
  \end{equation}

\medskip

    By Theorem \ref{t:BSU}, given an  $\alpha$-sequence $(\eps_j,n_j)$,  we have
    $f_{\eps_j}^{n_j} \to \lcal_\alpha$
    locally uniformly in $\mathcal B$.
    For simplicity of notation, we will fix an $\alpha$-sequence of the form $(\eps_n,n)$, so that
    the convergence takes the form
    \begin{equation}\label{e:lav-use}
    f_{\eps_n}^{n} \to \lcal_\alpha.
    \end{equation}
    For every $n$ and $i$, we denote by $g_{n,i}$
    the reading in the coordinates $(x,y)\in \D^2$ of the restriction of $f_{\epsilon_n}^{n m_i}$
    to $V_i$. We also denote by $g_n$ the holomorphic map which is equal to $g_{n_i}$ on $V_i$.
    Then, 
    in the coordinates $(x,y)$,
    for every $n$ sufficiently large, $g_n$ can be seen as 
    a horizontal-like map 
    \cite{DNS08dynamics}
     from a vertical subset to a horizontal subset of $\D\times \D$.
Up to slightly reducing the $U_i$ and 
$V$,
this vertical subset is actually
a subset of $\cup V_i$.
The horizontal subset consists of
$N$ small horizontal sets, contained in a small neighbourhood of $U=\{0\}\times \D$.

\medskip

It follows from \eqref{e:incl-L} and \eqref{e:lav-use}
that, for every $n$ sufficiently large and up to slightly reducing the $U_i$'s, we have
\begin{equation}
  \label{e:incl-g}
  g_{n,i}^{-1} (V_j) \Subset V_i
  \mbox{ for all } i,j  \quad 
  \mbox{ and }
  \quad g_n^{-1} (\cup V_i) \Subset \cup V_i.
  \end{equation}

Let us denote by $K^+_n$ the set of 
points in $\D^2$ whose orbit
under  $g_n$
never leaves $\D^2$
(which corresponds to point never leaving $V$ under appropriate iterates of $f_{\epsilon_n}^{n}$).
It follows from 
\eqref{e:incl-g} that,
for all 
$n$
sufficiently large,
$K^+_n$ is a collection of vertical graphs in $\D  \times\D$, which are
stable manifolds in $\D^2$.
Similarly, we can also
consider the
horizontal 
set $K^-_n$ which is a union of
unstable 
manifolds in $\D^2$.
The intersection $C_n := K_n^+ \cap K_n^-$ is
a Cantor set, on which the action of $f_{\epsilon_n}^n$
is uniformly hyperbolic. 
We denote by $\mu_n$ the measure of maximal entropy $\log N$ on $C_n$.
This measure admits two Lyapunov exponents
$\chi_n^-< 0< \chi_n^+$.
The exponent $\chi^+_n$ can be estimated
by means of the transversal contraction of the vertical sets $V_i$ under the inverse iteration of $g_n$. In particular, we deduce
the following upper bound
from \eqref{e:lav-use}
and the definition of $c$ in Proposition \ref{p:ifs}:

\begin{lem}\label{l:chin}
  Fix $\delta>1$ and let $c$ be as in Proposition \ref{p:ifs}.
  Then, for all $n \in \N$ large enough, we have
    $$\chi_n^+<    \delta \log c.$$
\end{lem}

Let $W^u(p)$ denote a generic unstable manifold of some $p \in C_n$,
and let $\nu_n$ denote the conditional measure of $\mu_n$ on $W^u(p)$.
It follows from \cite{LY2} 
that the Hausdorff dimension
of  $\nu_n$ is equal to $(\log N) / \chi_n^+$. 
As a consequence, we deduce from
Lemma \ref{l:chin} that, for every $\delta>1$,
the Hausdorff dimension of $\nu_n$ is larger than $\delta^{-1} \log N/c$ for every $n$ sufficiently large.

\medskip

Recall that through every $p \in C_n$ there is 
a stable manifold $W^s(p)$ which is a vertical graph in $\D\times \D$ and an unstable manifold $W^u(p)$ which is a horizontal graph in $\D\times \D$; in particular, $W^s(p)$ intersects every horizontal graph at exactly one point.
In particular, for every $y_0 \in \D$ there is a well-defined holonomy map
$\phi_{y_0}: \mathrm{supp}\, \nu_n \to \D$,
associated to the stable foliation of $g_n$
in $\D^2$, between the transversals $W^u(p)$ and the horizontal disk
$\{y=y_0\}$. By \cite{lyubichcft}
(see also \cite{LP21structure}),
the map $\phi_{y_0}$ is Lipschitz continuous. 
It follows 
that $K_n^+ = \bigcup_{y \in \D} \phi_y(\mathrm{supp}\, \nu_n)$
has Hausdorff dimension at least $2+  \delta^{-1} (2-\epsilon)$.
By 
the relation \eqref{e:Nc} between $N$ and $c$, this implies that we have
\[\dim_H (K^+_n) \geq 2 + \delta^{-1} (2-\epsilon)\]
for every $n$ sufficiently large.
Up to taking $\delta$ sufficiently close to $1$
and $n$ sufficiently large, we deduce the lower bound
$\dim_H (K^+_n) \geq 4-2\epsilon$. As the local coordinates $(x,y)$
are conformal
and the image of $K^+_n$ is contained in $J^+ (f_{\epsilon_n}^n)= J^+ (f_{\epsilon_n})$, this concludes the proof of Theorem \ref{th:maxdim}.

 \bibliographystyle{alpha} 
\bibliography{biblio-lavhen}

\newcommand{\etalchar}[1]{$^{#1}$}
\begin{thebibliography}{ABD{\etalchar{+}}16}

\bibitem[Aba01]{abate2001residual}
Marco Abate.
\newblock The residual index and the dynamics of holomorphic maps tangent to
  the identity.
\newblock {\em Duke Math. J.}, 107(1):173--207, 2001.

\bibitem[Aba05]{abate2005classification}
Marco Abate.
\newblock Holomorphic classification of 2-dimensional quadratic maps tangent to
  the identity.
\newblock {\em Surikaisekikenkyusho Kokyuroku}, (1447):1--14, 2005.

\bibitem[ABD{\etalchar{+}}16]{ABDPR16}
Matthieu Astorg, Xavier Buff, Romain Dujardin, Han Peters, and Jasmin Raissy.
\newblock A two-dimensional polynomial mapping with a wandering {F}atou
  component.
\newblock {\em Ann. of Math. (2)}, 184(1):263--313, 2016.

\bibitem[ABF21]{ABF}
Matthieu Astorg, Anna~Miriam Benini, and N{\'u}ria Fagella.
\newblock Bifurcation loci of families of finite type meromorphic maps.
\newblock {\em arXiv preprint arXiv:2107.02663}, 2021.

\bibitem[ABTP23]{ABP23}
Matthieu Astorg, Luka Boc~Thaler, and Han Peters.
\newblock Wandering domains arising from {L}avaurs maps with {S}iegel disks.
\newblock {\em Anal. PDE}, 16(1):35--88, 2023.

\bibitem[Ahl35]{Ahlfors1935theorie}
Lars Ahlfors.
\newblock Zur {T}heorie der \"{U}berlagerungsfl\"{a}chen.
\newblock {\em Acta Math.}, 65(1):157--194, 1935.

\bibitem[BC12]{BC12area}
Xavier Buff and Arnaud Ch\'{e}ritat.
\newblock Quadratic {J}ulia sets with positive area.
\newblock {\em Ann. of Math. (2)}, 176(2):673--746, 2012.

\bibitem[BD24]{BD24clt}
Fabrizio Bianchi and Tien-Cuong Dinh.
\newblock Every complex {H}\'{e}non map is exponentially mixing of all orders
  and satisfies the {CLT}.
\newblock {\em Forum Math. Sigma}, 12:Paper No. e4, 12, 2024.

\bibitem[Ber93]{Bergweiler93iteration}
Walter Bergweiler.
\newblock Iteration of meromorphic functions.
\newblock {\em Bull. Amer. Math. Soc. (N.S.)}, 29(2):151--188, 1993.

\bibitem[Ber98]{Bergweiler1998ahlfors}
Walter Bergweiler.
\newblock A new proof of the {A}hlfors five islands theorem.
\newblock {\em J. Anal. Math.}, 76:337--347, 1998.

\bibitem[Bia19]{bianchi19parabolic}
Fabrizio Bianchi.
\newblock Parabolic implosion for endomorphisms of {$\Bbb C^2$}.
\newblock {\em J. Eur. Math. Soc. (JEMS)}, 21(12):3709--3737, 2019.

\bibitem[BLS93a]{BLS93periodic}
Eric Bedford, Mikhail Lyubich, and John Smillie.
\newblock Distribution of periodic points of polynomial diffeomorphisms of
  {$\bold C^2$}.
\newblock {\em Invent. Math.}, 114(2):277--288, 1993.

\bibitem[BLS93b]{BLS93measure}
Eric Bedford, Mikhail Lyubich, and John Smillie.
\newblock Polynomial diffeomorphisms of {${\bf C}^2$}. {IV}. {T}he measure of
  maximal entropy and laminar currents.
\newblock {\em Invent. Math.}, 112(1):77--125, 1993.

\bibitem[Bow08]{Bowen75notes}
Rufus Bowen.
\newblock {\em Equilibrium states and the ergodic theory of {A}nosov
  diffeomorphisms}, volume 470 of {\em Lecture Notes in Mathematics}.
\newblock Springer-Verlag, Berlin, revised edition, 2008.
\newblock With a preface by David Ruelle.

\bibitem[BS91a]{BS91-BS1}
Eric Bedford and John Smillie.
\newblock Polynomial diffeomorphisms of {${\bf C}^2$}: currents, equilibrium
  measure and hyperbolicity.
\newblock {\em Invent. Math.}, 103(1):69--99, 1991.

\bibitem[BS91b]{BS91-BS2}
Eric Bedford and John Smillie.
\newblock Polynomial diffeomorphisms of {${\bf C}^2$}. {II}. {S}table manifolds
  and recurrence.
\newblock {\em J. Amer. Math. Soc.}, 4(4):657--679, 1991.

\bibitem[BS98]{BS98-BS5}
Eric Bedford and John Smillie.
\newblock Polynomial diffeomorphisms of {${\bf C}^2$}. {V}. {C}ritical points
  and {L}yapunov exponents.
\newblock {\em J. Geom. Anal.}, 8(3):349--383, 1998.

\bibitem[BSU17]{BSU17semi}
Eric Bedford, John Smillie, and Tetsuo Ueda.
\newblock Semi-parabolic bifurcations in complex dimension two.
\newblock {\em Comm. Math. Phys.}, 350(1):1--29, 2017.

\bibitem[CS15]{cheraghi2015satellite}
Davoud Cheraghi and Mitsuhiro Shishikura.
\newblock Satellite renormalization of quadratic polynomials.
\newblock {\em arXiv preprint arXiv:1509.07843}, 2015.

\bibitem[DH84]{DH84-p1}
Adrien Douady and John~H. Hubbard.
\newblock {\em \'{E}tude dynamique des polyn\^{o}mes complexes. {P}artie {I}},
  volume 84-2 of {\em Publications Math\'{e}matiques d'Orsay [Mathematical
  Publications of Orsay]}.
\newblock Universit\'{e} de Paris-Sud, D\'{e}partement de Math\'{e}matiques,
  Orsay, 1984.

\bibitem[DH85a]{DH85-p2}
Adrien Douady and John~H. Hubbard.
\newblock {\em \'{E}tude dynamique des polyn\^{o}mes complexes. {P}artie {II}},
  volume 85-4 of {\em Publications Math\'{e}matiques d'Orsay [Mathematical
  Publications of Orsay]}.
\newblock Universit\'{e} de Paris-Sud, D\'{e}partement de Math\'{e}matiques,
  Orsay, 1985.
\newblock With the collaboration of P. Lavaurs, Tan Lei and P. Sentenac.

\bibitem[DH85b]{DH85dynamics}
Adrien Douady and John~H. Hubbard.
\newblock On the dynamics of polynomial-like mappings.
\newblock {\em Ann. Sci. \'{E}cole Norm. Sup. (4)}, 18(2):287--343, 1985.

\bibitem[Din05]{Dinh05decay}
Tien-Cuong Dinh.
\newblock Decay of correlations for {H}\'{e}non maps.
\newblock {\em Acta Math.}, 195:253--264, 2005.

\bibitem[DL15]{DL15stability}
Romain Dujardin and Mikhail Lyubich.
\newblock Stability and bifurcations for dissipative polynomial automorphisms
  of {$\Bbb{C}^2$}.
\newblock {\em Invent. Math.}, 200(2):439--511, 2015.

\bibitem[DNS08]{DNS08dynamics}
Tien-Cuong Dinh, Vi\^{e}t-Anh Nguy\^{e}n, and Nessim Sibony.
\newblock Dynamics of horizontal-like maps in higher dimension.
\newblock {\em Adv. Math.}, 219(5):1689--1721, 2008.

\bibitem[Dou94]{Douady94Julia}
Adrien Douady.
\newblock Does a {J}ulia set depend continuously on the polynomial?
\newblock In {\em Complex dynamical systems ({C}incinnati, {OH}, 1994)},
  volume~49 of {\em Proc. Sympos. Appl. Math.}, pages 91--138. Amer. Math.
  Soc., Providence, RI, 1994.

\bibitem[DS06]{DS06geometry}
Tien-Cuong Dinh and Nessim Sibony.
\newblock Geometry of currents, intersection theory and dynamics of
  horizontal-like maps.
\newblock {\em Ann. Inst. Fourier (Grenoble)}, 56(2):423--457, 2006.

\bibitem[DS14]{DS14rigidity}
Tien-Cuong Dinh and Nessim Sibony.
\newblock Rigidity of {J}ulia sets for {H}\'{e}non type maps.
\newblock {\em J. Mod. Dyn.}, 8(3-4):499--548, 2014.

\bibitem[DSZ97]{DNS97}
Adrien Douady, Pierrette Sentenac, and Michel Zinsmeister.
\newblock Implosion parabolique et dimension de {H}ausdorff.
\newblock {\em C. R. Acad. Sci. Paris S\'{e}r. I Math.}, 325(7):765--772, 1997.

\bibitem[Duj04]{Duj04henon}
Romain Dujardin.
\newblock H\'{e}non-like mappings in {$\Bbb C^2$}.
\newblock {\em Amer. J. Math.}, 126(2):439--472, 2004.

\bibitem[{\'E}ca85]{Ecalle85-tome3}
Jean {\'E}calle.
\newblock {\em Les fonctions r\'{e}surgentes. {T}ome {III}}, volume 85-5 of
  {\em Publications Math\'{e}matiques d'Orsay [Mathematical Publications of
  Orsay]}.
\newblock Universit\'{e} de Paris-Sud, D\'{e}partement de Math\'{e}matiques,
  Orsay, 1985.
\newblock L'\'{e}quation du pont et la classification analytique des objects
  locaux. [The bridge equation and analytic classification of local objects].

\bibitem[EL92]{EL92entire}
Alexandre~\`E. Er\"{e}menko and Mikhail Lyubich.
\newblock Dynamical properties of some classes of entire functions.
\newblock {\em Ann. Inst. Fourier (Grenoble)}, 42(4):989--1020, 1992.

\bibitem[Eps93]{epstein1993towers}
Adam~L. Epstein.
\newblock {\em Towers of finite type complex analytic maps}.
\newblock PhD thesis, City University of New York, 1993.

\bibitem[FM89]{FM89dynamical}
Shmuel Friedland and John Milnor.
\newblock Dynamical properties of plane polynomial automorphisms.
\newblock {\em Ergodic Theory Dynam. Systems}, 9(1):67--99, 1989.

\bibitem[FS92]{FS92henon}
John~Erik Forn{\ae}ss and Nessim Sibony.
\newblock Complex {H}\'{e}non mappings in {${\bf C}^2$} and
  {F}atou-{B}ieberbach domains.
\newblock {\em Duke Math. J.}, 65(2):345--380, 1992.

\bibitem[HOV95]{HOV95}
John~H. Hubbard and Ralph~W. Oberste-Vorth.
\newblock H\'{e}non mappings in the complex domain. {II}. {P}rojective and
  inductive limits of polynomials.
\newblock In {\em Real and complex dynamical systems ({H}iller\o d, 1993)},
  volume 464 of {\em NATO Adv. Sci. Inst. Ser. C: Math. Phys. Sci.}, pages
  89--132. Kluwer Acad. Publ., Dordrecht, 1995.

\bibitem[Hub86]{Hubbard86Henon}
John~H. Hubbard.
\newblock The {H}\'{e}non mapping in the complex domain.
\newblock In {\em Chaotic dynamics and fractals ({A}tlanta, {G}a., 1985)},
  volume~2 of {\em Notes Rep. Math. Sci. Engrg.}, pages 101--111. Academic
  Press, Orlando, FL, 1986.

\bibitem[IS06]{inou2006renormalization}
Hiroyuki Inou and Mitsuhiro Shishikura.
\newblock The renormalization for parabolic fixed points and their
  perturbation.
\newblock {\em preprint}, 2006.

\bibitem[Jun42]{Jung42birationale}
Heinrich W.~E. Jung.
\newblock \"{U}ber ganze birationale {T}ransformationen der {E}bene.
\newblock {\em J. Reine Angew. Math.}, 184:161--174, 1942.

\bibitem[Lav89]{lavaurs1989systemes}
Pierre Lavaurs.
\newblock {\em Systemes dynamiques holomorphes: explosion de points
  p{\'e}riodiques paraboliques}.
\newblock PhD thesis, Paris 11, 1989.

\bibitem[LP21]{LP21structure}
Mikhail Lyubich and Han Peters.
\newblock Structure of partially hyperbolic {H}\'{e}non maps.
\newblock {\em J. Eur. Math. Soc. (JEMS)}, 23(9):3075--3128, 2021.

\bibitem[LY85]{LY2}
Fran\c{c}ois Ledrappier and Lai-Sang Young.
\newblock The metric entropy of diffeomorphisms. {II}. {R}elations between
  entropy, exponents and dimension.
\newblock {\em Ann. of Math. (2)}, 122(3):540--574, 1985.

\bibitem[Lyu99]{lyubichcft}
Mikhail Lyubich.
\newblock Feigenbaum-{C}oullet-{T}resser universality and {M}ilnor's hairiness
  conjecture.
\newblock {\em Ann. of Math. (2)}, 149(2):319--420, 1999.

\bibitem[McM00]{mcmullen}
Curtis~T. McMullen.
\newblock The {M}andelbrot set is universal.
\newblock In {\em The {M}andelbrot set, theme and variations}, volume 274 of
  {\em London Math. Soc. Lecture Note Ser.}, pages 1--17. Cambridge Univ.
  Press, Cambridge, 2000.

\bibitem[MR12]{lassevolker}
Volker Mayer and Lasse Rempe.
\newblock Rigidity and absence of line fields for meromorphic and {A}hlfors
  islands maps.
\newblock {\em Ergodic Theory Dynam. Systems}, 32(5):1691--1710, 2012.

\bibitem[Pes77]{Pesin77lyapunov}
Jakov~B. Pesin.
\newblock Characteristic {L}japunov exponents, and smooth ergodic theory.
\newblock {\em Uspehi Mat. Nauk}, 32(4(196)):55--112, 287, 1977.

\bibitem[PV20]{PV20}
Han Peters and Liz Vivas.
\newblock Parabolic implosion.
\newblock {\em Notices Amer. Math. Soc.}, 67(8):1095--1103, 2020.

\bibitem[RR12]{lasserippon}
Lasse Rempe and Philip~J. Rippon.
\newblock Exotic {B}aker and wandering domains for {A}hlfors islands maps.
\newblock {\em J. Anal. Math.}, 117:297--319, 2012.

\bibitem[Shi98]{Shishikura98hausdorff}
Mitsuhiro Shishikura.
\newblock The {H}ausdorff dimension of the boundary of the {M}andelbrot set and
  {J}ulia sets.
\newblock {\em Ann. of Math. (2)}, 147(2):225--267, 1998.

\bibitem[Shi00]{Shishikura00bifurcation}
Mitsuhiro Shishikura.
\newblock Bifurcation of parabolic fixed points.
\newblock In {\em The {M}andelbrot set, theme and variations}, volume 274 of
  {\em London Math. Soc. Lecture Note Ser.}, pages 325--363. Cambridge Univ.
  Press, Cambridge, 2000.

\bibitem[Sib99]{Sibony99}
Nessim Sibony.
\newblock Dynamique des applications rationnelles de {$\bold P^k$}.
\newblock In {\em Dynamique et g\'{e}om\'{e}trie complexes ({L}yon, 1997)},
  volume~8 of {\em Panor. Synth\`eses}, pages 97--185. Soc. Math. France,
  Paris, 1999.

\bibitem[Tan98]{Tan98hausdorff}
Lei Tan.
\newblock Hausdorff dimension of subsets of the parameter space for families of
  rational maps. ({A} generalization of {S}hishikura's result).
\newblock {\em Nonlinearity}, 11(2):233--246, 1998.

\bibitem[Ued86]{Ueda86local1}
Tetsuo Ueda.
\newblock Local structure of analytic transformations of two complex variables.
  {I}.
\newblock {\em J. Math. Kyoto Univ.}, 26(2):233--261, 1986.

\bibitem[Ued91]{Ueda91local2}
Tetsuo Ueda.
\newblock Local structure of analytic transformations of two complex variables.
  {II}.
\newblock {\em J. Math. Kyoto Univ.}, 31(3):695--711, 1991.

\bibitem[Vor81]{Voronin81analytic}
S.~M. Voronin.
\newblock Analytic classification of germs of conformal mappings {$({\bf
  C},\,0)\rightarrow ({\bf C},\,0)$}.
\newblock {\em Funktsional. Anal. i Prilozhen.}, 15(1):1--17, 96, 1981.

\bibitem[Zin98]{Zinsmeister98}
Michel Zinsmeister.
\newblock Fleur de {L}eau-{F}atou et dimension de {H}ausdorff.
\newblock {\em C. R. Acad. Sci. Paris S\'{e}r. I Math.}, 326(10):1227--1232,
  1998.

\end{thebibliography}

\end{document}